\documentclass[12pt]{article}
\usepackage[margin=1.25in]{geometry} 
\usepackage{amsmath,amsthm,amssymb,amsfonts}
\usepackage{graphicx}
\usepackage{xfrac}
\usepackage{mathabx}
\usepackage{enumitem}
\usepackage{physics}
\usepackage{dirtytalk}

\mathchardef\ordinarycolon\mathcode`\:
\mathcode`\:=\string"8000
\begingroup \catcode`\:=\active
  \gdef:{\mathrel{\mathop\ordinarycolon}}
\endgroup
\usepackage{tikz}
\usetikzlibrary{calc,arrows}

\newcommand{\ts}{\text{  }}

\newtheorem{theorem}{Theorem}[section]

\newtheorem{lemma}[theorem]{Lemma}

\usepackage{dsfont}

\usepackage{hyperref}
\hypersetup{
    colorlinks=true,
    linkcolor=blue,
    filecolor=magenta,      
    urlcolor=cyan,
    }

\title{Null-Controllability of a Non-Local Heat Equation}
\author{Steven Walton}
\date{}

\begin{document}
   
\maketitle
\begin{abstract}
    We consider the null-controllability of a non-local heat equation by interior $L^2(\Omega)$ controls. We confirm a conjecture of Lissy and Zuazua by showing that it is enough to assume that the kernel $k(x,\xi)$ is symmetric and $k(x,\xi)\in L^2(\Omega\times\Omega)$ in order to obtain the result.  The result is obtained by treating the non-local linear operator $K:L^2(\Omega)\rightarrow L^2(\Omega)$ as a compact and bounded perturbation of the Dirichlet-Laplacian, $A$, which leads to useful bounds on the semigroup generated by $(A+K,\mathcal{D}(A))$.  Using this approach, we are also able to estimate the cost of control. 
\end{abstract}

\tableofcontents
\newpage

\section{Introduction}
Let $\Omega \subset \mathbb{R}^d$ with $\partial \Omega \in C^2$ and $\omega \subset\subset \Omega$ with boundary the same regularity as $\Omega$.  Let $T>0$, we use the notation $Q = (0,T)\times \Omega$, $\partial_x Q = (0,T)\times \partial\Omega$ and $Q\big|_{t=t'} = \{t'\}\times \Omega$ for any $t'\in[0,T]$. We consider the null-controllability of a non-local heat equation, 
\begin{equation*}
\tag{NLH}
\begin{aligned}
    (\partial_t - \Delta)u(x,t) = f\chi(\omega) +\int_\Omega k(x,\xi)u(\xi,t) d\xi & \text{ in }Q\\
    u(x,t) = 0 & \text{ on } \partial_x Q\\
    u(x,0) = u_0 & \text{ in } Q\Big|_{t=0}
\end{aligned}
\label{eqn:NLH}
\end{equation*}
We will assume that  $k(x,\xi) = k(\xi,x) \in L^2(\Omega\times\Omega)$.
\par
We have in both cases $f=f(x,t) \in L^2((0,T)\times\omega)$, $u_0\in L^2(\Omega)$ and $\chi(\omega)$ is the characteristic function of the interior subset $\omega$ and we wish to show that for controls $f$,
\begin{equation}
    \tag{NullCond}
    \begin{aligned}
    \| f\|^2_{L^2((0,T);\omega)}\leq \kappa_T \| u_0 \|^2_{L^2(\Omega)}\\
    \| u(x,T) \|_{L^2(\Omega)} = 0
    \end{aligned}
    \label{eqn:NullCond}
\end{equation}
That is the state is driven to zero in time $T$ by the controls $f$.  The constant $\kappa_T$ is the cost of control and will have the form 
\[
 \kappa_T = C(t,T) e^{C/T^\alpha}
\]
which blows up as $T\rightarrow 0^+$.  
\par
We will use a by now classical approach to demonstrate null-controllability by verifying the adjoint of \ref{eqn:NLH} satisfies an observability inequality
\begin{equation}\tag{FinalObs}
\begin{aligned}
    \|\varphi(0)\|^2_{L^2(\Omega)} \leq  \kappa_T \| \varphi\|^2_{L^2(0,T;\omega)}
\end{aligned}
\label{eqn:FinalObs}
\end{equation}
where $\varphi = \varphi(x,t)$ denotes the solution to the adjoint equation in either case with differing cost of control.
\par
Before we outline our strategy for verifying \ref{eqn:FinalObs}, we would like to briefly recall previously obtained results to similar problems.

\section{Background and Known Results}
Before discussing equations with a non-local term, we outline a few of the known results for parabolic equations. There are two common approaches to showing null-controllability, the Lebeau-Robbiano strategy based on local Carleman estimates and that of Imanouvilov and Fursikov based on global Carleman estimates.
\par  
The former strategy was first developed in \cite{LR95} showing directly the controllability of the heat equation.  These results were extended in \cite{LebZua98} and \cite{JerLeb99}.  We will mention these works again later upon introducing the spectral observability inequality \ref{eqn:SpecObs}.
\par
The latter strategy can be found in \cite{ImanFurs96}.  These estimates rely on suitable choice of weight functions to derive the Carleman inequalities.  
\par
The trouble begins upon introduction of the non-local term.  The global Carleman estimate for \ref{eqn:NLH}  would look like (\cite{FCLZ16} equation 9)
\begin{equation*}
    \| \rho \varphi \|^2_{L^2(0,T;\Omega)} \leqslant  C(\epsilon)\| \rho\varphi \|^2_{L^2(0,T;\omega)} + \epsilon \|\rho K\varphi \|^2_{L^2(0,T;\Omega)}
\end{equation*}
Where $\epsilon >0$ can be arbitrarily small and $\rho = \rho(x,t) \rightarrow 0$ as $t\rightarrow T$ is the weight function as mentioned above. Here and in the sequel we will write the non-local linear operator $K: L^2(\Omega) \rightarrow L^2(\Omega)$ for the non-local term, e.g. for $w\in L^2(\Omega)$
\begin{equation}
\tag{NLoc}
    Kw = \int_\Omega k(x,\xi) w (\xi)d\xi
\label{eqn:NLoc}
\end{equation}
The above Carleman estimate is a pointwise estimate and thus the non-local term is unable to be absorbed by the right hand side.
\par
To the author's knowledge, the first to tackle \ref{eqn:NLH} are the authors of \cite{FCLZ16}, in which they assume $K$ to be analytic and apply a compactness-uniqueness argument to obtain the observability inequality.  Later in \cite{LisZua18}, the results are extended to parabolic systems with the same assumption on $K$.  In neither case is an estimate for the cost of control obtained.  For the 1-d case, \cite{MicTak18} assume the kernel to be degenerate, $k=g(x)h(\xi)$, and are able to obtain a spectral inequality via the formation of a Riesz basis.  The authors in \cite{BicH-S19} allow for time dependence of $k$ and assume an exponential decay condition holds for $K$ which enforces a compatibility of the non-local term with the weight functions $\rho$ in the Carleman estimate.    
\par
Let us quote from \cite{LisZua18},
\say{Hence, a natural conjecture would be that the main result ... holds under the assumption that $A(x, \xi) \in \mathcal{M}_n(L^2(\Omega\times\Omega))...$} Where $A$ is a matrix version of our $k$. The main goals of this paper is to give a positive answer to the above conjecture, and therefore the only assumptions we make on the kernel $k$ have already been stated, namely that it is symmetric and a member of $L^2(\Omega\times\Omega)$.
\par
Throughout what follows, $A=\Delta$ is the Dirichlet-Laplacian, $K$ is as in equation \ref{eqn:NLoc}, $\mathcal{D}(A) = H^2(\Omega)\cap H^1_0(\Omega)$ and $L=A+K$ with $\mathcal{D}(L)=\mathcal{D}(A)$.  We will simply write $A$ and $L$ for the generators $(A,\mathcal{D}(A))$ and $(L,\mathcal{D}(A))$, respectively. We write the adjoint equations to \ref{eqn:NLH} in the abstract formulation (using $t\mapsto T-t$), 
\begin{equation}
    \tag{Adj}
    \begin{aligned}
        (\partial_t - L)\varphi = 0 \text{ in } Q\\
        \varphi(0) = \varphi_0 \text{ in } Q\Big|_{t=0}
    \end{aligned}
    \label{eqn:Adj}
\end{equation}
Which has solution
\[
 \varphi(t) = e^{L(T-t)}\varphi_0,
\]
as will be verified later once the properties of the semigroup, $(e^{L(T-t)})_{t\geq 0}$, generated by $L$ have been established. Thus, we may rewrite \ref{eqn:FinalObs} as
\begin{equation}
    \tag{SemObs}
    \| e^{LT}\varphi_0\|^2_{L^2(\Omega)} \leqslant \kappa_T \| e^{L(T-t)}\varphi_0\|^2_{L^2(0,T;\omega)}
    \label{eqn:SemObs}
\end{equation}
 which will be the inequality we seek to verify. For simplicity, we will write simply $t$ for $T-t$, since we will not deal with the primal problem again.
\subsubsection*{Outline}

In the next section, we recall some definitions and results from the theory of semigroups and show that $L$ generates an analytic semigroup and possesses a left inverse.  We also detail the spectral observability inequality previously mentioned.  With these preliminaries in hand, we proceed with the proof of the observability inequality \ref{eqn:SemObs} and cost estimate simultaneously.


\section{Preliminaries}
 Here and elsewhere, when it is unambiguous, we write $\| \cdot \|$ for $\|\cdot \|_{\mathcal{L}(L^2(\Omega))}$ and $\|\cdot\|_{L^2(\Omega)}$. If $X$ is a subset of $\Omega$, then $\|\cdot\|_X$ denotes the restriction of the above norms to that subset.
\begin{lemma}
    The operator $K$ is compact and bounded.
\end{lemma}
\begin{proof}
    From our assumption on the kernel, $k$, we have by definition that $K$ is a Hilbert-Schmidt operator and thus continuous and compact.
\end{proof}

The following is a well-known result and may be found in \cite{Kue19} Chapter 14, for example.
\begin{theorem}\label{anaA}
    The semigroup generated by $A$ is analytic and satisfies, for all $t\geq 0$ and constants $C,\beta > 0$,
    \begin{equation*}
    \begin{aligned}
        \| e^{A t} \| \leqslant C e^{\beta t} 
    \end{aligned}
    \end{equation*} 
    
\end{theorem}
Due to the self-adjointness of $L$, and the fact that $K$ is compact and bounded, we may conclude that $L$ is a sectorial operator also, and thus $L$ generates an analytic semigroup.  The estimate in the following theorem is proved in \cite{Dav07}, Theorem 11.4.1.
\begin{theorem}\label{analytic_expLt}
 The semigroup generated by $L$ is analytic and satisfies
 \begin{equation}
     \tag{SemEst}
     \begin{aligned}
       \| e^{L t} \| \leqslant C e^{(\beta + C\|K\|) t} 
     \end{aligned}
     \label{eqn:SemEst}
 \end{equation}
 with $C,\beta$ the same as in Theorem \ref{anaA}.
\end{theorem}
Our last theorem concerning the semigroup $(e^{Lt})_{t \geq 0}$ establishes the existence of a left-inverse. It is well known that the heat semigroup is not exactly observable, except in the case of the trivial subset. However, this is exactly the scenario we need for $(e^{Lt})_{t \geq 0}$ in the computations of the next section.  Namely, we only need that $\| \varphi_0 \|^2_{L^2(\omega)} \leq C(t)\| e^{Lt}\varphi_0 \|^2_{L^2(\omega)}$ in one line of the computations to follow.  We add these remarks to emphasize that we are {\bf not} claiming exact observability from $\mathcal{L}(L^2(\Omega),L^2(\omega))$, which would not in general be true even for $(e^{At})_{t\geq 0}$. \\

\begin{theorem}\label{inverse} The semigroup $(e^{Lt})_{t\geq 0}$ possesses left-inverses in $\mathcal{L}(L^2(\omega))$.  That is, there exists a constant, $\zeta(t)$, such that for every $\upsilon \in L^2(\Omega)$
\begin{equation}
    \tag{SemInv}
    \zeta(t)\|\upsilon\|_\omega \leqslant \| e^{Lt}\upsilon \|_\omega.
    \label{SemInv}
\end{equation}
\end{theorem}
To prove the above theorem, we use Theorem 2 of \cite{Zwa13} which we repeat here.
\begin{theorem}\label{Zwart}
 The semigroup generated by $S$ possesses a left-inverse if and only if $-S$ can be extended to the generator of a $C_0-$semigroup.
\end{theorem}
\begin{proof}[Proof of Theorem \ref{inverse}]
The same arguments applied to $L$ to show that $L$ generates an analytic semigroup may be applied to $-L$ to show that the assumptions of  Theorem \ref{Zwart} hold.  
\end{proof}
Next, we recall the (beautiful) spectral observability inequality first derived in \cite{LR95}, and given in the present form in \cite{LebZua98} and \cite{JerLeb99}.  It has seen a plethora of applications since.  We use the version found in Theorem 2.3 of \cite{FuLuZha19}.
\begin{theorem}
 For any sequence $\{ c_j\}\subset \mathbb{C}$ and for every $r>0$ the following holds
 \begin{equation}
     \tag{SpecObs}
     \sum_{\lambda_j \leqslant r}|c_j|^2 \leqslant Ce^{C\sqrt{r}}\| \sum_{\lambda_j \leqslant r} c_j \psi_j \|^2_{L^2(\omega)}
     \label{eqn:SpecObs}
 \end{equation}
 where $\lambda_j$ and $\psi_j$ denote the jth eigenvalue and eigenfunction of $A$, respectively. 
\end{theorem}

\section{Proof of the Observability Inequalities}
\begin{proof}
We start by estimating the left hand side of \ref{eqn:SemObs}. From \ref{eqn:SemEst} we have
\begin{equation}\label{lhs_est}
    \begin{aligned}
    \| e^{LT}\varphi_0\|^2 \leq Ce^{(\beta + C\|K\|)T}\|\varphi_0\|^2
    \end{aligned}
\end{equation}
where the constants have been adjusted to take into account the square.  By assumption, $\varphi_0\in L^2(\Omega)$ and thus may be represented as a series expansion in the $(\lambda_j,\psi_j)$,
\begin{equation*}
    \varphi_0 = \sum_{j\geq 1} \varphi_{0,j} \psi_j
\end{equation*}
and then 
\[
\|\varphi_0\|^2 = \sum_{j\geq 1}|\varphi_{0,j}|^2
\]
To use \ref{eqn:SpecObs}, we must have that $\varphi_0\in E_r$, the span of the eigenfunctions corresponding to $\sigma_r(A)=\{\lambda \in \sigma(A), \ts \lambda \leq r \}$. For the moment, we make this assumption.  
\par
With this in mind, we use the mentioned eigenfunction expansion  in (\ref{lhs_est}) and applying \ref{eqn:SpecObs} we get (constants may change from line to line)
\begin{equation}\label{sem_spec_est}
    \begin{aligned}
    \| e^{LT}\varphi_0\|^2 \leq Ce^{(\beta + C\|K\|)T}\sum_{\lambda_j \leq r}|\varphi_{0,j}|^2 \\
    \leq Ce^{(\beta+C\| K\|)T+C_1\sqrt{r}}\| \sum_{\lambda_j \leq r}\varphi_{0,j}\psi_j \|^2_{L^2(\omega)} \\
    =
    Ce^{(\beta+C\| K\|)T+C_1\sqrt{r}}\| \varphi_0 \|^2_{L^2(\omega)}
    \end{aligned}
\end{equation}
Now we apply the left-inverse property \ref{SemInv} to the last line of \ref{sem_spec_est},
\begin{equation}\label{sem_inv_est}
    \begin{aligned}
Ce^{(\beta+C\| K\|)T+C_1\sqrt{r}}\| \varphi_0 \|^2_{L^2(\omega)} \leq C_2(t)e^{(\beta+C\| K\|)T+C_1\sqrt{r}}\| e^{Lt}\varphi_0 \|^2_{L^2(\omega)}
\end{aligned}
\end{equation}
Thus, combining \ref{lhs_est} with \ref{sem_inv_est} we get
\begin{equation}
    \| e^{LT}\varphi_0\|^2_{L^2(\Omega)} \leq C_2(t)e^{(\beta+C\| K\|)T+C_1\sqrt{r}}\| e^{Lt}\varphi_0 \|^2_{L^2(\omega)}
\end{equation}
Then, integrating from $0$ to $T$, choosing $r$ so that $r = \frac{1}{T}$ and writing $C_3 = \beta+C\| K\|$ we obtain
\begin{equation}\label{result1}
\begin{aligned}
\| e^{LT}\varphi_0\|^2_{L^2(\Omega)}\leq  \frac{C_2(t)}{T}e^{C_3 T+C_1/\sqrt{T}}\| e^{Lt}\varphi_0  \|^2_{L^2(0,T;\omega)}
\end{aligned}
\end{equation}
Now the higher frequency modes may be controlled by following the well known strategy in \cite{LR95}.  
\end{proof} 

\section{Conclusion and Outlook}
In this note, we use results from the theory of semigroups generated by compact perturbations of the Dirichlet-Laplacian to derive a null-controllability result for a non-local heat equation.  We confirm a conjecture of Lissy and Zuazua which asserts that control may be gotten if the kernel is square integrable in both arguments.  Our results rely heavily on the spectral properties of the Dirichlet-Laplacian. 
\par 
We would like to adapt some of the theory presented here and use the techniques put forth in \cite{Tran13} to study hyperbolic equations with non-local terms, especially of kinetic type. Indeed, a similar approach has been taken for the linearized Boltzmann equation in \cite{LeauBoltz}, however we would like to investigate the controllability of kinetic equations such as the Smoluchowski coagulation equation and the 3-wave kinetic equation.  

\end{document}